\documentclass[12pt,reqno]{article}
\usepackage{graphicx}
\usepackage{amssymb}
\usepackage{amsmath}
\usepackage{amsthm}
\usepackage{amsfonts}
\usepackage{url}
\usepackage{hyperref}
\usepackage{breakurl}

\newcommand{\N}{{\mathbb N}}

\newtheorem{theorem}{Theorem}
\newtheorem{corollary}{Corollary}

\newtheorem{definition}{Definition}
\newtheorem{remark}{Remark}

\begin{document}
\bibliographystyle{plain}
\title{~\\[-20pt] 
Bounds on minors of binary matrices}
\author{Richard P. Brent\\
Australian National University\\
Canberra, ACT 0200,
Australia\\
\href{mailto:minors@rpbrent.com}{\tt minors@rpbrent.com}
\and
Judy-anne H. Osborn\\
The University of Newcastle\\
Callaghan, NSW 2308,
Australia\\
\href{mailto:Judy-anne.Osborn@newcastle.edu.au}%
{\tt Judy-anne.Osborn@newcastle.edu.au}
}

\date{}	

\maketitle
\thispagestyle{empty}                   

\begin{abstract}
We prove an upper bound on sums of squares of minors of $\{+1, -1\}$
matrices. The bound is sharp for Hadamard matrices, a result
due to de Launey and Levin (2009), but our proof is simpler.
We give several corollaries relevant to minors of Hadamard matrices,
and generalise a result of Tur\'an on determinants of random $\{+1,-1\}$
matrices.
\end{abstract}

\pagebreak[3]
\section{Introduction}		\label{sec:intro}

A $\{+1, -1\}$-matrix (abbreviated ``$\{\pm 1\}$-matrix'' below) is
a matrix $A$ whose elements are $+1$ or $-1$. We consider $n\times n$
$\{\pm 1\}$-matrices; $n$ is called the \emph{order} of the matrix.
A \emph{minor of order $m$} is the determinant of an $m\times m$ submatrix
$M$ of $A$.

Theorem~\ref{thm:mean_detsq} gives an upper bound on the mean square
of the minors of order $m$ of any $\{\pm 1\}$ matrix $A$
of order $n \ge m$.  The upper bound is attained if $A$ is a Hadamard
matrix, and this case was proved by de Launey and
Levin~\cite[Proposition~2]{LL}. Our proof, using the Cauchy-Binet
formula~\cite{Gantmacher,Muir}, 
is much simpler than the proof given for the Hadamard case 
by de Launey and Levin,
which requires consideration of the cycle structure of random
permutations and an identity involving Stirling numbers.

In \S\ref{sec:cor}
we give several easy corollaries of Theorem~\ref{thm:mean_detsq}. 

Corollary~\ref{cor:m_factorial} shows that, in the mean square sense, the
minors of Hadamard matrices are strictly larger than the minors of random
$\{\pm1\}$-matrices, except for the trivial case of minors of order~$1$.

A difficult, not yet completely solved, problem is to find the 
asymptotic behaviour of the
probability that a random $\{\pm1\}$-matrix of order~$n$ is
singular, see~\cite{Kahn,TV}.
In Corollary~\ref{cor:ineq_YZ} we consider a simpler but analogous problem
concerning zero minors of $\{\pm1\}$-matrices.
The corollary  gives a lower bound on the number of zero 
minors of order $m$ of a $\{\pm1\}$ matrix of order~$n$.  The bound is
nontrivial in the cases $2 \le m \le 6$.

Corollary~\ref{cor:singular_submats} gives a criterion for when a
$\{\pm1\}$ matrix must have singular minors of small order,
and a lower bound on their number. In some cases
the result is sharper than that obtained by a standard argument
using Dirichlet's ``pigeon-hole'' principle.

Corollary~\ref{cor:small_minors} gives exact
formul{\ae} for the number of zero minors of orders $2$ and $3$ in Hadamard
matrices.  The formula for minors of order~$2$ is implicit in a paper of
Little and Thuente~\cite{LT}, but the result for minors of order~$3$ 
appears to be new.

Finally, Theorem~\ref{thm:Turan-gen1} generalises a well-known result
of Tur\'an~\cite{Turan} on the mean-square determinant of a random
$\{\pm1\}$-matrix.

For simplicity, in \S\S\ref{sec:data}--\ref{sec:cor} we consider only
minors of square $\{\pm1\}$-matrices. The results can be extended
without difficulty to minors of rectangular matrices, 
say $n \times p$ $\{\pm1\}$-matrices 
with minors of order $m \le \min(n,p)$.
It is also possible to extend some of the results to rectangular submatrices
$M$, say $m\times m'$, where $m\le m'$, if
$|\!\det(M)|^2$ is replaced by $\det(MM^T)$.

\section{The mean square of minors}\label{sec:data}

Theorem~\ref{thm:mean_detsq} gives
an upper bound on the mean square of minors of $\{\pm 1\}$
matrices. The bound is sharp because it is attained for Hadamard matrices.
For the case that the matrix $A$ is a Hadamard matrix, the result
is due to de Launey and Levin~\cite[Proposition~2]{LL},
and their proof could perhaps be modified to show that strict inequality 
occurs when $A$ is not a Hadamard matrix. However,
we give a different and simpler proof.
\begin{definition}
If $A$ is a $\{\pm 1\}$ matrix and $m \in \N$, then
$S_m(A)$ is the set of all $m \times m$ submatrices of $A$.
\end{definition}
\begin{theorem}	\label{thm:mean_detsq}
Let $A$ be a square $\{\pm 1\}$ matrix of order $n \ge m > 1$.
Then the mean value
$E(\det(M)^2) $
of $\det(M)^2$, taken over all $M \in S_m(A)$, satisfies
\begin{equation}	\label{eq:mean_sq_upperbd}
E(\det(M)^2) \le n^m\Big/\binom{n}{m}\,.
\end{equation}
Moreover, equality holds in~\eqref{eq:mean_sq_upperbd} 
iff $A$ is a Hadamard matrix.
\end{theorem}

\pagebreak[3]
\begin{proof}
Consider the $m \times n$ matrix $B$ formed by taking
any $m$ rows of $A$, and apply the Cauchy-Binet formula to $B$,
obtaining
\begin{equation}	\label{eq:sum_over_subset}
\det(BB^T) = \sum_{M \in S_m(B)} \det(M)^2.
\end{equation}
From Hadamard's inequality~\cite{Hadamard}, the
left side of~\eqref{eq:sum_over_subset} is bounded above by $n^m$,
with equality occurring iff the rows of $B$ are orthogonal.
Thus
\[\sum_{M \in S_m(B)} \det(M)^2 \le n^m.\]
Summing over all $\binom{n}{m}$ ways in which we can choose $B$,
we obtain
\[\sum_{M \in S_m(A)} \det(M)^2 \le n^m \binom{n}{m}\,.\]
Now, dividing by $|S_m(A)| = \binom{n}{m}^2$ to give the mean value
over all submatrices of order $m$, we obtain the
inequality~\eqref{eq:mean_sq_upperbd}. It is clear from the proof
that equality occurs in~\eqref{eq:mean_sq_upperbd} iff the
rows of $B$ are pairwise orthogonal for all choices of $B$. 
Since $m \ge 2$, this 
implies that the rows of $A$ are pairwise orthogonal, and hence
$A$ is a Hadamard matrix.
\end{proof}

\section{Corollaries}	\label{sec:cor}

Tur\'an~\cite{Turan} showed that the expected value of $\det(A)^2$ for 
$\{\pm1\}$ matrices $A$ of order $m$, chosen uniformly at random, is
$m!$.
Corollary~\ref{cor:m_factorial} shows that, for submatrices $M$ of 
Hadamard matrices, the mean value
of $\det(M)^2$ is always greater
than the expected value for random $\{\pm1\}$ matrices,
excluding the trivial case $m=1$ for which equality occurs.
\begin{corollary}	\label{cor:m_factorial}
Let $H$ be a Hadamard matrix of order $n \ge m > 1$.
Then the mean value $E(\det(M)^2)$
of $\det(M)^2$, taken over all $M \in S_m(H)$, satisfies
\[E(\det(M)^2) > m!\]
\end{corollary}
\begin{proof}
From Theorem \ref{thm:mean_detsq},
\[
E(\det(M)^2) = n^m\Big/\binom{n}{m} =
m!\prod_{k=1}^{m-1}\left(1-\frac{k}{n}\right)^{-1} >\; m!
\]
\end{proof}

\begin{definition}
Let $A$ be a square $\{\pm 1\}$ matrix of order $n \ge m \ge 1$.
Then
\begin{center}
$Z(m,A)$ is the number of zero minors of order~$m$ of $A$, and\\
$Y(m,A)$ is the number of nonzero minors of order~$m$ of $A$. 
\end{center}
\end{definition}

\begin{corollary} 	\label{cor:ineq_YZ}
Let $A$ be a square $\{\pm 1\}$ matrix of order $n \ge m > 1$.
Then 
\[Y(m,A) \le 4\left(\frac{n}{4}\right)^m\binom{n}{m}\] 
and
\[Z(m,A) \ge \binom{n}{m}\left\{\binom{n}{m} - 
  4\left(\frac{n}{4}\right)^m\right\}.
\]
Moreover, if $m \le 3$, then equality occurs iff $A$ is a Hadamard
matrix.
\end{corollary}
\begin{proof}
Using a well-known mapping from $\{\pm1\}$ matrices of order $m$
to $\{0,1\}$-matrices of order $m-1$,
it is easy to prove that
the determinant of an order~$m$ $\{\pm1\}$ matrix is divisible by $2^{m-1}$.
Thus, each nonzero minor of order~$m$ has square at least $4^{m-1}$,
and
\begin{equation}	\label{eq:ineq_Y1}
\sum_{M\in S_m(A)} \det(M)^2 \ge 4^{m-1}Y(m,A).
\end{equation}
However, from Theorem~\ref{thm:mean_detsq} we have
\[\sum_{M\in S_m(A)} \det(M)^2 \le n^m\binom{n}{m}\,.\]
Thus $4^{m-1}Y(m,A) \le n^m\binom{n}{m}$, which gives the inequality
for $Y(m,A)$.  The inequality for $Z(m,A)$ follows from the observation
that
\[
Y(m,A) + Z(m,A) = \binom{n}{m}^2,
\]
since the total number of minors of order $m$ is $\binom{n}{m}^2$.
Finally, suppose that $1 < m \le 3$. Then there is only one
nonzero value of $\det(M)^2$, namely $4^{m-1}$. Thus, equality occurs
in~\eqref{eq:ineq_Y1}, and the last sentence of the corollary follows from
the last sentence of Theorem~\ref{thm:mean_detsq}.
\end{proof}

Corollary~\ref{cor:singular_submats} shows that a sufficiently large
$\{\pm1\}$ matrix always has singular submatrices of order $m \le 6$.
In fact, such submatrices occur with positive density at least $p_m$,
where $p_m$ is given in Table~$1$.
\begin{table}[hb]
\begin{center}
\begin{tabular}{|c|c|c|c|c|}
\hline
$m$	& $p_m$	& $\widehat{p}_m$	& $n_0(m)$	& $2^{m-1}+1$\\
\hline
2	& 0.5000 & 0.5000 &	3	&	3\\
3	& 0.6250 & 0.6250 &	5	&	5\\
4	& 0.6250 & 0.5898 &	8	&	9\\
5	& 0.5312 & 0.5001 &	15	&	17\\
6	& 0.2969 & 0.3924 &	45	&	33\\
\hline
\end{tabular}
\caption{Lower bound on zero minor probability $p_m$,}
and threshold $n_0(m)$, see Corollary~\ref{cor:singular_submats}.\\
For $p_m$, $\widehat{p}_m$ see eqns.~\eqref{eq:p_m}--\eqref{eq:p_mhat}.
\end{center}
\end{table}
\begin{corollary}	\label{cor:singular_submats}
Let $A$ be a $\{\pm 1\}$ matrix of order $n$, and suppose $2 \le m \le 6$.
Then $A$ has a singular submatrix of order $m$ if $n \ge n_0(m)$,
where $n_0(m)$ is as in Table~$1$.
\end{corollary}
\begin{proof}
$A$ has a singular submatrix of
order $m$ iff $Z(m,A) > 0$, and from Corollary~\ref{cor:ineq_YZ} a
sufficient condition for this is that
\begin{equation}	\label{eq:ineq_fac}
\binom{n}{m} > 4\left(\frac{n}{4}\right)^m.
\end{equation}
Since $m! < 4^{m-1}$ for $2 \le m \le 6$ 
we see that~\eqref{eq:ineq_fac} holds for $2 \le m \le 6$
provided that $n$ is sufficiently large.  In fact, a computation shows
that we need $n \ge n_0(m)$, where $n_0(m)$ is given in Table~1.
\end{proof}
\begin{remark}	\label{remark:pigeonhole}
{\rm 
Consider $A$ as in Corollary~\ref{cor:singular_submats}.
If $n > 2^m$ then, by Dirichlet's ``pigeonhole'' principle, 
any $m\times n$ submatrix $B$ of $A$ must have two
identical columns, so $A$ must have a singular $m \times m$ submatrix.
In fact, by normalising the first row of $B$ to be 
$(+1,+1,\ldots,+1)$, 
the statement is true for $n > 2^{m-1}$. Comparing
$n_0(m)$ and $2^{m-1}+1$ (see Table~$1$), we see that 
Corollary~\ref{cor:singular_submats} gives a slightly stronger result
for $m \in \{4,5\}$. Also, the proof of Corollary~\ref{cor:singular_submats}
shows that the density of singular submatrices as $n \to \infty$
is at least
\begin{equation}	\label{eq:p_m}
p_m =
 \lim_{n\to\infty}\left(1-4\left(\frac{n}{4}\right)^m\Big/\binom{n}{m}\right)
  = 1 - 4^{1-m}m!\,.
\end{equation} 
Using an extension of the argument above that used the pigeonhole principle,
we obtain a corresponding density
\begin{equation}	\label{eq:p_mhat}
\widehat{p}_m = 1 - \prod_{k=1}^{m-1}(1-2^{1-m}k)\,.
\end{equation}
Table~$1$ gives the values of $\widehat{p}_m$ for $2 \le m \le 6$
to 4 decimal places;
we see that $p_m > \widehat{p}_m$ for 
$4 \le m \le 5$.	
\hspace*{\fill}[End of remark~\ref{remark:pigeonhole}.]
}
\end{remark}

The frequencies of small singular submatrices of Hadamard matrices are given
in Corollary~\ref{cor:small_minors}.
The corollary is restricted to $m \le 3$ because
for $m>3$ we find by computation that
$Z(m,H)$ depends on the Hadamard equivalence class of $H$.
For example, this is true if $n=16$ and $m=4$,
when there are four possible values of $Z(m,H)$. 
It is straightforward to prove
Corollary~\ref{cor:small_minors} by enumeration of
the singular submatrices of order $m \in \{2,3\}$, 
but we give a shorter proof using Corollary~\ref{cor:ineq_YZ}.

\begin{corollary}	\label{cor:small_minors}
Let $H$ be a Hadamard matrix of order $n$. Then
\begin{equation}	\label{eq:Z2}
Z(2,H) = n^2(n-1)(n-2)/8,  \;\;\text{and}
\end{equation}
\begin{equation}	\label{eq:Z3}
Z(3,H) = n^2(n-1)(n-2)(n-4)(5n-4)/288.
\end{equation}
\end{corollary}
\begin{proof}
This is just the last part of Corollary~\ref{cor:ineq_YZ}, where
we have explicitly computed and simplified the expressions
for $Z(m,H)$ in the cases $m=2$ and $m=3$.
\end{proof}

\pagebreak[3]
\begin{remark}
{\rm
We expect random $\{\pm 1\}$-matrices of order $2$ to be singular with
probability $1/2$, and matrices of order $3$ to be singular 
with probability $5/8$,
see~\cite{Kahn,OEIS}. 
These probabilities agree with the limiting probabilities that we obtain from
Corollary~\ref{cor:small_minors} as $n \to \infty$. More
precisely,
\[
Z(2,H)\Big/\binom{n}{2}^2 = \frac12 - O\left(\frac{1}{n}\right) 
\;\text{ and }\;
Z(3,H)\Big/\binom{n}{3}^2 = \frac{5}{8} - O\left(\frac{1}{n}\right)\,.
\]
In this sense the minors of order $2$ and $3$ of Hadamard matrices
of order~$n$
behave like the minors of random $\{\pm1\}$ matrices in the limit
as $n\to\infty$.
}
\end{remark}
\begin{remark}
{\rm
From
Sz\"oll\H{o}si's theorem~\cite[Proposition 5.5]{Szollosi10}
or Jacobi's determinant identity~\cite{BS,Jacobi},
\[Z(m,H) = Z(n-m,H).\] 
Thus, the minors
of order $m \ge n-3$ of Hadamard matrices of order~$n$ take only a small
number of distinct values and certainly do \emph{not} 
behave like the minors of random 
$\{\pm 1\}$-matrices as $n \to \infty$.
Previously, such results were obtained by a more detailed study
of the structure of Hadamard matrices,
see for example~\cite{KLMS03, KMS01, LL, Sharpe07}. 
}
\end{remark}

\section{Generalisation of a result of Tur\'an}

The following theorem generalises the result of Tur\'an~\cite{Turan}
mentioned in \S\ref{sec:cor}.

\begin{theorem}	\label{thm:Turan-gen1}
If $B \in \{\pm1\}^{m \times n}$ is chosen uniformly at random, then
\[E(\det(BB^T)) = m!\binom{n}{m}.\]
\end{theorem}
\begin{proof}
The proof uses the Cauchy-Binet theorem much
as in the proof of Theorem~\ref{thm:mean_detsq}. We write
\[\det(BB^T) \;\; = \sum_{M \in S_m(B)} \det(M)^2,\]
where $S_m(B)$ is the set of all $m \times m$ submatrices of $B$,
and take expectations.  There are $\binom{n}{m}$ choices of $M$,
and by Tur\'an's theorem each choice contributes $m!$ to the expectation.
\end{proof}
\begin{remark}
{\rm
Changing notation,
de Launey and Levin~\cite[Proposition 2]{LL}
says that, if $H$ is a Hadamard
matrix of order $h$, and $B$ is chosen uniformly at random from the 
$n \times m$ submatrices of $H$, then
\[E(\det(B^T B)) = h^m \binom{n}{m}\Big/\binom{h}{m}.\]
The right-hand-side tends to $m!\binom{n}{m}$ as $h \to \infty$.
Thus, in this sense, fixed-size submatrices of large Hadamard matrices
tend to behave like random matrices.
}
\end{remark}

\end{document}